\documentclass[10pt, a4paper]{article}
\usepackage{}

\usepackage{lineno}

\usepackage{graphicx}
\usepackage{ae}
\usepackage{amsmath}
\usepackage{amssymb}
\usepackage{latexsym}
\usepackage{url}
\usepackage{epsfig}

\usepackage{cite}
\usepackage{mathrsfs}
\usepackage{amsfonts}
\usepackage{amsthm}
\usepackage{indentfirst}

\def\qed{\hfill$\Box$\vspace{12pt}}

\usepackage{color}

\input amssym.def
\input amssym.tex

\long\def\delete#1{}

\newcommand{\pmat}[1]{\begin{pmatrix}#1\end{pmatrix}}

\newcommand{\dmat}[1]{\begin{vmatrix}#1\end{vmatrix}}

\newcommand{\be}{\begin{equation}}
\newcommand{\ee}{\end{equation}}
\newcommand{\bea}{\begin{eqnarray}}
\newcommand{\eea}{\end{eqnarray}}
\newcommand{\bean}{\begin{eqnarray*}}
\newcommand{\eean}{\end{eqnarray*}}

\def\deg{{\rm deg}}

\newtheorem{thm}{Theorem}[section]
\newtheorem{cor}[thm]{Corollary}
\newtheorem{lem}[thm]{Lemma}
\newtheorem{prop}[thm]{Proposition}

\numberwithin{equation}{section}

\title{Spectral characterizations of propeller graphs}

\author{Xiaogang Liu \quad and \quad Sanming Zhou
\\
{\small Department of Mathematics and Statistics}\\
{\small The University of Melbourne}\\
{\small Parkville, VIC 3010, Australia}\\
{\small xiaogliu@student.unimelb.edu.au, smzhou@ms.unimelb.edu.au}}

\date{}

\begin{document}

\openup 0.5\jot
\maketitle
%\linenumbers

\begin{abstract}
A \emph{propeller graph} is obtained from an $\infty$-graph by attaching a path to the vertex of degree four, where an $\infty$-graph consists of two cycles with precisely one common vertex. In this paper, we prove that all propeller graphs are determined by their Laplacian spectra as well as their signless Laplacian spectra.

\bigskip


\bigskip

\end{abstract}

\section{Introduction}

All graphs considered in the paper are undirected and simple. Let $G=(V(G),E(G))$ be a
graph with vertex set $V(G)=\{v_1,v_2,\ldots,v_n\}$ and edge set
$E(G)$. The \emph{adjacency matrix} of $G$, denoted by $A(G)$, is the $n \times n$ matrix whose $(i,j)$-entry is $1$ if $v_i$ and $v_j$ are adjacent and $0$ otherwise. Denote by $d_i=d_G(v_i)$ the degree of $v_i$ in $G$, and by
$$
\deg(G) = (d_1,d_2,\ldots,d_n)
$$
the degree sequence of $G$. The \emph{Laplacian matrix} of $G$ is defined as $L(G)=D(G)-A(G)$, where $D(G)$ is the diagonal matrix with diagonal entries $d_1,d_2,\ldots,d_n$. We call $Q(G)=D(G)+A(G)$ the \emph{signless Laplacian matrix} of $G$.
Denote the eigenvalues of $A(G)$, $L(G)$ and $Q(G)$ by $\lambda_1\geq\lambda_2\geq\cdots\geq\lambda_n$,
$\mu_1\geq\mu_2\geq\cdots\geq\mu_n$ and $\nu_1\geq\nu_2\geq\cdots\geq\nu_n$, respectively.
The collection of eigenvalues of $A(G)$ together with multiplicities are called the \emph{$A$-spectrum} of $G$. Two graphs are said to be \emph{$A$-cospectral} if they have the same $A$-spectrum. A graph is called an \emph{$A$-DS graph} if it is \emph{determined by its $A$-spectrum}, meaning that there exists no other graph that is non-isomorphic to it but $A$-cospectral with it. Similar terminology will be used for $L(G)$ and $Q(G)$. So we can speak of \emph{$L$-spectrum}, \emph{$Q$-spectrum}, \emph{$L$-cospectral graphs}, \emph{$Q$-cospectral graphs}, \emph{$L$-DS graphs} and \emph{$Q$-DS graphs}.

Which graphs are determined by their spectra? This is a classical question in spectral graph theory which was raised by G\"{u}nthard and Primas \cite{kn:Gunthard56} in 1956 with motivations from chemistry. This problem is also related to complexity theory. It is well-known that the complexity of the problem of determining graph isomorphism is unknown \cite{kn:Garey79}. Since checking whether two graphs are cospectral can be done in polynomial time, the isomorphism problem can be reduced to the one of checking isomorphism between cospectral graphs. Up to now, many graphs have been proved to be determined by their ($A$, $L$ or/and $Q$) spectra \cite{kn:Boulet08,kn:Boulet09,kn:Cvetkovic95,kn:Cvetkovic10,kn:vanDam03,kn:vanDam09,kn:Haemers08,kn:Lu09,kn:LiuM10,kn:Wangliu10,kn:Mirzakhah10,kn:Omidi07,kn:Wang10,kn:Zhou12}. However, the problem of determining $A$-DS (respectively, $L$-DS, $Q$-DS) graphs is still far from being completely solved. Therefore, finding new families of DS graphs deserves further attention in order to enrich our database of DS graphs. Unfortunately, even for some simple-looking graphs, it is often challenging to determine whether they are $A$-DS, $L$-DS or $Q$-DS.
\begin{figure}
\centering
\includegraphics[height=6cm]{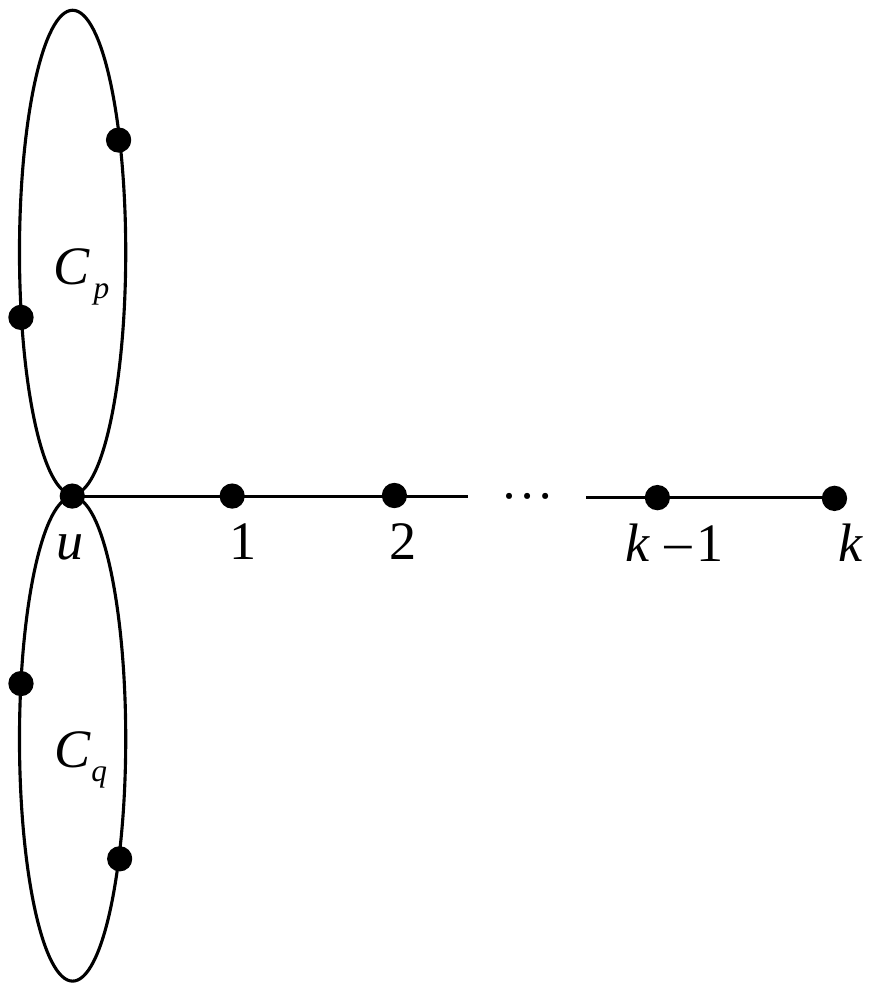}
\vspace{-0.8cm}
\caption{\small A propeller graph.}
\label{f1}
\end{figure}

In this paper we give a new family graphs that are both $L$-DS and $Q$-DS. We define a \emph{propeller graph} (see Fig. \ref{f1}) as a graph obtained from an $\infty$-graph by attaching a path to the vertex of degree 4, where an \emph{$\infty$-graph} is a graph consisting of two cycles with exactly one vertex in common \cite{kn:Wang10}. The main results of this paper are as follows.

\begin{thm}\label{thm1}
All propeller graphs are determined by their $L$-spectra.
\end{thm}

\begin{thm}\label{thm2}
All propeller graphs are determined by their $Q$-spectra.
\end{thm}

Since the $L$-spectrum of a graph determines that of its complement \cite{kn:Kelmans65}, Theorem \ref{thm1} implies that the complement of any propeller graph is also determined by its $L$-spectrum.

We will prove Theorems \ref{thm1} and \ref{thm2} in Sections \ref{sec:L} and \ref{sec:Q}, respectively.

\section{Preliminaries}

In this section we collect some known results that will be used in the proof of Theorems \ref{thm1} and \ref{thm2}. Denote by
$$
\phi(M)=\phi(M;x)=\det(xI-M)=l_0x^n+l_1x^{n-1}+\dots+l_n
$$
the characteristic polynomial of an $n \times n$ matrix $M$, where $I$ is
the identity matrix of the same size. In particular, for a graph $G$, we call $\phi(A(G))$  (respectively, $\phi(L(G))$, $\phi(Q(G))$) the \emph{adjacency} (respectively, \emph{Laplacian}, \emph{signless Laplacian}) \emph{characteristic polynomial} of $G$.

Denote by $n_3(G)$ the number of triangles in $G$.

\begin{lem}\label{Lcoefficients}\emph{\cite{kn:Oliveira02}}
Let $G$ be a graph with $n$ vertices and $m$ edges, and let
$\mathrm{deg}(G)=(d_1,d_2,\dots ,d_n)$ be its degree
sequence. Then the first four coefficients in $\phi(L(G))$ are:
\[
    l_0=1,\qquad l_1=-2m,\qquad l_2=2m^2-m-\frac{1}{2}\sum_{i=1}^{n}d_i^2,
\]
\[
    l_3=\frac{1}{3}\left(-4m^3+6m^2+3m\sum_{i=1}^{n}d_i^2-\sum_{i=1}^{n}d_i^3-3\sum_{i=1}^{n}d_i^2+6n_3(G)
    \right).
\]
\end{lem}

The following result follows from \cite{kn:vanDam03} and Lemma \ref{Lcoefficients}.

\begin{lem}\label{spectrum}
Let $G$ be a graph. The following can be determined by its $L$-spectrum:
\begin{itemize}
\item[\rm (a)] the number of vertices of $G$;
\item[\rm (b)] the number of edges of $G$;
\item[\rm (c)] the number of components of $G$;
\item[\rm (d)] the number of spanning trees of $G$.
\end{itemize}
\end{lem}

\begin{lem}\label{cycle}\emph{\cite{kn:Cvetkovic95}}
Let $u$ be a vertex of $G$, $N(u)$ the set of vertices of $G$ adjacent to $u$, and $C(u)$ the set of cycles of $G$ containing $u$. Then
\begin{eqnarray*}
\phi(A(G);x)&=&x \phi(A(G-u);x)-\sum_{v\in N(u)}\phi(A(G-u-v);x)-2\sum_{Z\in C(u)}\phi(A(G - V(Z));x).
\end{eqnarray*}
\end{lem}

\begin{lem}\label{Ldegrees}\emph{\cite{kn:Wang10}}
Let $G$ be a graph with $n$ vertices, $m$ edges and degree sequence $\mathrm{deg}(G)=(d_1,d_2,\ldots,d_n)$. If a graph $H$ with degree sequence $\mathrm{deg}(H)=(d_1+t_1,d_2+t_2,\dots,d_n+t_n)$ is $L$-cospectral (respectively, $Q$-cospectral) with $G$, then $t_1,t_2,\dots,t_n$ are integers such that
\[\sum_{i=1}^nt_i=0~~ \text{and}~~\sum_{i=1}^n(t_i^2+2d_it_i)=0.\]
\end{lem}

Denote by $P_n$ and $C_n$ the path and cycle on $n$ vertices,
respectively. Let $B_n$ be the matrix of order $n$ obtained from
$L(P_{n+1})$ by deleting the row and column corresponding to one
end vertex of $P_{n+1}$, and $U_n$ be the matrix of order $n$
obtained from $L(P_{n+2})$ by deleting the rows and columns
corresponding to the two end vertices of $P_{n+2}$.

\begin{lem}\label{Lpolynomial1}\emph{\cite{kn:Guo08}}
Set $\phi(L(P_0))=0$, $\phi(B_0)=1$, $\phi(U_0)=1$. Then
\begin{itemize}
\item[\rm (a)] $\phi(L(P_{n+1}))=(x-2)\phi(L(P_n))-\phi(L(P_{n-1})),\ (n\ge 1)$;
\item[\rm (b)] $x\phi(B_n)=\phi(L(P_{n+1}))+\phi(L(P_n))$;
\item[\rm (c)] $\phi(L(P_n))=x\phi(U_{n-1}),\ (n\ge 1)$;
\item[\rm (d)] $\phi(L(C_n))=\frac{1}{x}\phi(L(P_{n+1}))-\frac{1}{x}\phi(L(P_{n-1}))+2(-1)^{n+1},\
(n\ge3)$.
\end{itemize}
\end{lem}

Combining these and $\phi(L(P_1);4)=4$, we obtain the following formulas.

\begin{prop}\label{eigenvalue4}
\noindent\emph{(a)}  $ \phi(L(P_n);4)=4n;$  ~~\emph{(b)} $\phi(B_n;4)=2n+1;$  ~~\emph{(c)} $\phi(U_{n};4)=n+1;$  ~~\emph{(d)} $\phi(L(C_n);4)=2+2(-1)^{n+1}.$
\end{prop}

For a vertex $v$ of $G$, let $L_v(G)$ denote the principal sub-matrix of $L(G)$ formed by deleting the row and column corresponding to $v$.

\begin{lem}\label{Lpolynomial2}\emph{\cite{kn:Guo05}}
Let $G_1$ and $G_2$ be vertex-disjoint graphs. Let $G$ be the graph obtained by taking the union of $G_1$ and $G_2$ and then adding an edge between a vertex $u$ of $G_1$ and a vertex $v$ of $G_2$. Then
\[
\phi(L(G))=\phi(L(G_1))\phi(L(G_2))-\phi(L(G_1))\phi(L_v(G_2))-\phi(L(G_2))\phi(L_u(G_1)).
\]
\end{lem}

\begin{lem}\label{degrees}\emph{\cite{kn:Cvetkovic07,kn:Simic07}}
Let $G$ be a graph with $n$ vertices, $m$ edges and $n_3(G)$
triangles. Let $T_k=\sum_{i=1}^{n}\nu_i^k$ be the $k$th $Q$-spectral moment of $G$,
$k=0,1,2,\ldots$. Then
\begin{eqnarray*}
   & &T_0=n,~~T_1=\sum_{i=1}^{n}d_i=2m,~~T_2=2m+\sum_{i=1}^{n}d_i^2,~~T_3=6n_3(G)+3\sum_{i=1}^{n}d_i^2+\sum_{i=1}^{n}d_i^3.
\end{eqnarray*}
\end{lem}

From Lemma \ref{degrees}, we can easily get the following result.

\begin{lem}\label{Qtriangles}
Let $G$ and $H$ be $Q$-cospectral graphs. Then
\begin{itemize}
\item[\rm (a)] $G$ and $H$ have the same number of vertices;
\item[\rm (b)] $G$ and $H$ have the same number of edges;
\item[\rm (c)] $\sum\limits_{v \in V(G)}d_G(v)^2=\sum\limits_{v \in V(H)}d_H(v)^2$;
\item[\rm (d)] $6n_3(G)+\sum\limits_{v \in V(G)}d_G(v)^3=6n_3(H)+\sum\limits_{v \in V(H)}d_H(v)^3$.
\end{itemize}
\end{lem}

Let $\mathcal {L}(G)$ denote the line graph of a graph $G$. Let
$\mathcal {S}(G)$ be the \emph{subdivision graph} of $G$ obtained by
replacing each edge of $G$ by a path of length two. The
$Q$-spectrum of a graph can be exactly expressed by the $A$-spectrum
of its line and subdivision graphs
\cite{kn:Cvetkovic07,kn:Cvetkovic08,kn:Cvetkovic09}, and the
following results can be found in
\cite{kn:Cvetkovic07,kn:Cvetkovic08,kn:Wang10}.

\begin{lem}\label{AQAcospectral} If two graphs $G$ and $H$ are
$Q$-cospectral, then $\mathcal {L}(G)$ and
$\mathcal {L}(H)$ are $A$-cospectral.
\end{lem}

\begin{lem}\label{QA}
Two graphs $G$ and $H$ are $Q$-cospectral if and
only if $\mathcal {S}(G)$ and $\mathcal {S}(H)$ are $A$-cospectral.
\end{lem}

\begin{lem}\label{closedfourwalk}\emph{\cite{kn:Cvetkovic87}}
Let $G$ be a graph with $n$ vertices and $m$ edges. Let $n_4(G)$
be the number of subgraphs of $G$ isomorphic to $C_4$, and $x_k$ the number of vertices of degree $k$ in $G$. Then
\begin{eqnarray*}
\sum\limits_i\lambda_i^4&=&8n_4(G)+\sum\limits_kkx_k+4\sum_{k\geq2}\frac{k(k-1)}{2}x_k.
\end{eqnarray*}
\end{lem}

A spanning subgraph of $G$ whose components are trees or odd-unicyclic graphs is called a \emph{$TU$-subgraph} of $G$ \cite{kn:Cvetkovic07}. Suppose that a $TU$-subgraph $G^{TU}$ of $G$ contain $c$ unicyclic graphs and trees $T_1,T_2,\ldots,T_s$. The weight $W(G^{TU})$ of  $G^{TU}$  is defined by
\[W(G^{TU})=4^c\prod_{i=1}^s(1+|E(T_i)|).\]
Then the coefficients of $\phi(Q(G))$ can be expressed in terms of the weights of $TU$-subgraphs of $G$ as follows.

\begin{lem}\label{QTUgraph}\emph{\cite{kn:Cvetkovic07}}
Let $\phi(Q(G))=q_0x^n+q_1x^{n-1}+\dots+q_n$. Then $q_0=1$ and
\[q_j=\sum_{G^{TU}_j}(-1)^jW(G^{TU}_j),\quad j=1,2,\ldots,n,\]
where the summation runs over all $TU$-subgraphs $G^{TU}_j$ of $G$ with $j$ edges.
\end{lem}

\section{Proof of Theorem \ref{thm1}}
\label{sec:L}

Throughout this section we use $G$ to denote a propeller graph with $n=p+q+k-1$ vertices as shown in Fig. \ref{f1}. To prove Theorem \ref{thm1}, we first compute the Laplacian characteristic polynomial of $G$. Before proceeding, we need the following results.

\begin{prop}\label{LLLLLpolynomialL}
Let $G_1$ and $G_2$ be vertex-disjoint graphs. Let $G_1\cdot G_2$ be the coalescence obtained from $G_1$ and $G_2$ by identifying a vertex $u$ of $G_1$ with a vertex $v$ of $G_2$. Then
\begin{eqnarray*}
  \phi(L(G_1\cdot G_2);x)&=&\phi(L(G_1))\phi(L_v(G_2))+\phi(L_u(G_1))\phi(G_2)-x\phi(L_u(G_1))\phi(L_v(G_2)).
\end{eqnarray*}
\end{prop}

\begin{proof}
The coalescence $G_1\cdot G_2$ has Laplacian matrix $$\pmat{
                                              L_u(G_1)    & \mathrm{\mathbf{u}} &  O \\
                                              \mathrm{\mathbf{u}}^T &  d_{G_1}(u)+d_{G_2}(v)     &  \mathrm{\mathbf{v}} \\
                                                O^T         &  \mathrm{\mathbf{v}}^T            &  L_v(G_2)
                                                },$$
where
$\pmat{
L_u(G_1)    & \mathrm{\mathbf{u}}   \\
\mathrm{\mathbf{u}}^T &  d_{G_1}(u)
}$
and
$\pmat{
d_{G_2}(v)     &  \mathrm{\mathbf{v}} \\
\mathrm{\mathbf{v}}^T  & L_v(G_2)
}$ are the Laplacian matrices of $G_1$ and $G_2$ respectively, and $O$ is the zero matrix of appropriate size. Then
\[
\begin{aligned}
\phi(L(G_1\cdot G_2);x)
&=\dmat{
                            xI-L_u(G_1)            & -\mathrm{\mathbf{u}}           &  O \\
                            -\mathrm{\mathbf{u}}^T &  x-d_{G_1}(u)-d_{G_2}(v)     &  -\mathrm{\mathbf{v}} \\
                            O^T                    &  -\mathrm{\mathbf{v}}^T        &  xI-L_v(G_2)}\\
&=\dmat{
                            xI-L_u(G_1)            & -\mathrm{\mathbf{u}}           &  O \\
                            -\mathrm{\mathbf{u}}^T &  x-d_{G_1}(u)                  &  -\mathrm{\mathbf{v}} \\
                            O^T                    & \mathrm{\mathbf{0}}            &  xI-L_v(G_2)}\\
&\quad+\dmat{
                            xI-L_u(G_1)            & \mathrm{\mathbf{0}}            &  O \\
                            -\mathrm{\mathbf{u}}^T &  x-d_{G_2}(v)                  &  -\mathrm{\mathbf{v}} \\
                            O^T                    & -\mathrm{\mathbf{v}}^T         &  xI-L_v(G_2)}\\
&\quad+\dmat{
                            xI-L_u(G_1)            & \mathrm{\mathbf{0}}            &  O \\
                            -\mathrm{\mathbf{u}}^T & -x                             &  -\mathrm{\mathbf{v}} \\
                            O^T                    & \mathrm{\mathbf{0}}            &  xI-L_v(G_2)},
\end{aligned}
\]and the result follows.
\qed\end{proof}

\begin{prop}
Let $G_{p,q}$ be an $\infty$-graph consisting of cycles $C_p$ and $C_q$ with a common vertex $u$. Then
\begin{eqnarray}
  \phi(L(G_{p,q});x)&=&(x-4)\phi(U_{p-1})\phi(U_{q-1})-2\phi(U_{q-1})\left(\phi(U_{p-2})+(-1)^{p}\right)-2\phi(U_{p-1})\left(\phi(U_{q-2})+(-1)^{q})\right),\nonumber\\
                    &&\label{Propellerpolynomial}\\
   \phi(L(G_{p,q});4)&=&2(p+q)-4pq-2\left((-1)^q p+(-1)^p q\right).\label{Propellerpolynomial2}
\end{eqnarray}
\end{prop}
\begin{proof}
Lemma \ref{Lpolynomial1} implies that
\begin{eqnarray}
\phi(L(C_n))&=&\frac{1}{x}\phi(L(P_{n+1}))-\frac{1}{x}\phi(L(P_{n-1}))+2(-1)^{n+1}\nonumber\\
            &=&\frac{1}{x}\big((x-2)\phi(L(P_n))-\phi(L(P_{n-1}))\big)-\frac{1}{x}\phi(L(P_{n-1}))+2(-1)^{n+1}\nonumber\\
            &=&\frac{x-2}{x}\phi(L(P_n))-\frac{2}{x}\phi(L(P_{n-1}))+2(-1)^{n+1}\nonumber\\
            &=&(x-2)\phi(U_{n-1})-2\phi(U_{n-2})+2(-1)^{n+1}.\label{Neweqa121}
\end{eqnarray}
Note that $G_{p,q}$ is a coalescence of $C_p$ and $C_q$. Thus we obtain (\ref{Propellerpolynomial}) by using (\ref{Neweqa121}), $\phi(L_u(C_{q})) = \phi(U_{q-1})$, $\phi(L_u(C_{p})) = \phi(U_{p-1})$ and Proposition \ref{LLLLLpolynomialL}. (\ref{Propellerpolynomial2}) is an immediate consequence of (\ref{Propellerpolynomial}) and Proposition \ref{eigenvalue4}.
\qed\end{proof}

\begin{prop}\label{propeller4}
Let $G$ be a propeller graph with $n=p+q+k-1$ vertices as shown in Fig. \ref{f1}. Then
\begin{eqnarray}
\phi(L(G);x) &=& \phi(L(G_{p,q}))\phi(L(P_k))-\phi(L(G_{p,q}))\phi(B_{k-1})-\phi(L(P_k))\phi(U_{p-1})\phi(U_{q-1}),\label{propellerExpanding}\\
\phi((L(G);4)&=&2(2k+1)\left(p+q-(-1)^q p-(-1)^p q\right)-4pq(3k+1).\nonumber
\end{eqnarray}
\end{prop}
\begin{proof}
We obtain (\ref{propellerExpanding}) by using Lemma \ref{Lpolynomial2} and $\phi(L_u(G_{p,q})) = \phi(U_{p-1})\phi(U_{q-1})$.
From (\ref{Propellerpolynomial2}) and (\ref{propellerExpanding}) and Proposition \ref{eigenvalue4}, we have
\begin{eqnarray*}
\phi(L(G);4)&=&\phi(L(G_{p,q});4)\phi(L(P_k);4)-\phi(L(G_{p,q});4)\phi(B_{k-1};4)-\phi(L(P_k);4)\phi(U_{p-1};4)\phi(U_{q-1};4)\\
            &=&\left(2(p+q)-4pq-2((-1)^q p+(-1)^p q)\right)(4k-(2k-1))-4kpq\\
           &=&2(2k+1)\left(p+q-(-1)^q p-(-1)^p q\right)-4pq(3k+1)
\end{eqnarray*}
as required.
\qed\end{proof}

Note that $\phi(L(P_{n+1}))=(x-2)\phi(L(P_n))-\phi(L(P_{n-1}))$ by Lemma \ref{Lpolynomial1}. Solving this recurrence equation, and noting $\phi(L(P_0)) = 0$ and $\phi(L(P_1)) = x$, we obtain that, for $n\ge 1$,
\begin{eqnarray}\label{LaplacianPolwhole}
          \phi(L(P_n))&=&\frac{(y+1)(y^{2n}-1)}{y^{n+1}-y^n},
        \end{eqnarray}
where $y$ satisfies the characteristic equation $y^2-(x-2)y+1=0$ with $x\neq4$. Substituting (\ref{LaplacianPolwhole}) into (b) and (c) of Lemma \ref{Lpolynomial1}, we obtain
\begin{eqnarray}
          \phi(B_n)&=&\frac{y^{2n+1}-1}{y^{n+1}-y^n},\label{LaplacianPolBn}\\
          \phi(U_n)&=&\frac{y^{2n+2}-1}{y^{n+2}-y^{n}}.\label{LaplacianPolHn}
        \end{eqnarray}
Plugging (\ref{LaplacianPolwhole}), (\ref{LaplacianPolBn}) and (\ref{LaplacianPolHn}) into (\ref{Propellerpolynomial}) and then (\ref{propellerExpanding}), and with the help of Maple, we obtain
\begin{eqnarray}\label{Ppolywhole}
 y^{n}(y-1)^3(y+1)^2\phi(L(G))+1-3y-4y^2+4y^{2n+3}+3y^{2n+4}-y^{2n+5}&=&f_L(p,q,k;y),
\end{eqnarray}
where
\begin{eqnarray*}
\begin{array}{llll}
   f_L(p,q,k;y)=& 2(-1)^{1+q}y^{2p+q+2k+3}    &+2(-1)^{1+p}y^{2q+p+2k+3}   &+2(-1)^qy^{2p+q+2k+1} \\
                &+2(-1)^py^{p+2q+2k+1}        &+3y^{2p+2q+1}+3y^{2p+2q}    &+y^{2p+3+2k}           \\
                &+y^{2q+3+2k}                 &+3y^{2p+2k+2}               &+3y^{2q+2k+2}          \\
                &+2y^{2p+1+2k}                &+2y^{2q+1+2k}               &+2(-1)^qy^{2p+2+q}\\
                &+2(-1)^py^{2q+2+p}           &+2(-1)^{1+q}y^{2p+q}        &+2(-1)^{1+p}y^{2q+p}\\
                &+2(-1)^py^{3+p+2k}           &+2(-1)^qy^{3+q+2k}          &+2(-1)^{1+p}y^{p+2k+1}\\
                &+2(-1)^{1+q}y^{q+2k+1}       &-2y^{2p+2}-2y^{2q+2}        &-3y^{2p+1}-3y^{2q+1}\\
                &-y^{2p}-y^{2q}-3y^{2k+3}     &+2(-1)^{1+p}y^{2+p}         &+2(-1)^{1+q}y^{2+q}\\
                &+2(-1)^py^{p}                &+2(-1)^qy^{q}               &-3y^{2k+2}.
\end{array}
\end{eqnarray*}

\begin{lem}\label{lem4}
No two non-isomorphic propeller graphs are $L$-cospectral.
\end{lem}

\begin{proof}
Let $G$ and $G'$ be $L$-cospectral propeller graphs with $n=p+q+k-1$ and $n'=p'+q'+k'-1$ vertices, respectively. Without loss of generality, we let $p\geq q$ and $p'\geq q'$. By (a) and (d) of Lemma \ref{spectrum}, we have
\begin{eqnarray}
p+q+k&=&p'+q'+k'.\label{lem4E01}\\
pq&=&p'q'.\label{lem4E02}
\end{eqnarray}
By (\ref{Ppolywhole}), we then get
\begin{eqnarray}\label{lem4Epropeller4}
f_L(p,q,k;y)=f_L(p',q',k';y).
\end{eqnarray}
Clearly, the term in $f_L(p,q,k;y)$ with the smallest exponent is $2(-1)^qy^{q}$ or $-3y^{2k+2}$, and similarly for $f_L(p',q',k';y)$. From (\ref{lem4Epropeller4}) we have either $2(-1)^qy^{q}=2(-1)^{q'}y^{q'}$ or $-3y^{2k+2}=-3y^{2k'+2}$. In the former case, we have $q=q'$, and so $p=p'$ and $k=k'$ by (\ref{lem4E01}) and (\ref{lem4E02}). In the latter case, we have $k=k'$, and so $(p,q)=(p',q')$ by (\ref{lem4E01}) and (\ref{lem4E02}). Therefore, $G$ and $G'$ are isomorphic in each case.
\qed\end{proof}

\begin{lem}
\label{lem1}
Let $H$ be a graph that is $L$-cospectral with the propeller graph $G$. Then
\[\mathrm{deg}(H) = (5,2^{n-2},1),(4^2,2^{n-4},1^2),(4,3^3,2^{n-7},1^3),\;\mbox{or}\,\;(3^6,2^{n-10},1^4),\]
where the exponent denotes the number of vertices in $H$ having the corresponding degree.
\end{lem}

\begin{proof}
Suppose $\mathrm{deg}(H)=(5+t_1,2+t_2,2+t_3,\dots,2+t_{n-1},1+t_n)$. Since $\mathrm{deg}(G)=(5,2^{n-2},1)$ and $H$ is $L$-cospectral with $G$, by (c) of Lemma \ref{spectrum},
\begin{equation}\label{bounds}
t_1 \ge -4,\, t_2 \ge -1,\, \ldots,\, t_{n-1} \ge -1,\, t_n \ge 0.
\end{equation}
Moreover, by Lemma \ref{Ldegrees}, $t_1,t_2,\dots,t_n$ are integers such that
\begin{eqnarray}
&&\sum_{i=1}^nt_i=0, \label{lem1E01}\\
&&\sum_{i=1}^nt_i^2+4\sum_{i=2}^{n-1}t_i+10t_1+2t_n=0.\label{lem1E02}
\end{eqnarray}
So $t_1=-\sum_{i=2}^{n-1}t_i-t_n$. Plugging this into (\ref{lem1E02}) yields
\begin{equation}
\label{lem1E04}
t_1^2+6t_1+a=0,
\end{equation}
where $a$ is given by
\begin{equation}
\label{a}
\sum_{i=2}^{n-1}t_i^2 = a - (t_n^2-2t_n).
\end{equation}
Obviously, $a \ge t_n^2-2t_n \ge -1$.
Solving (\ref{lem1E04}) for $t_1$, we get
\begin{eqnarray}
t_1&=&-3\pm\sqrt{9-a}.\label{lem1E05}
\end{eqnarray}
Since $t_1$ is an integer and $-1\leq a\leq9$, we see that $a=0,5,8,9$. We discuss these cases one by one.

\medskip
\noindent\emph{Case 1.} $a=0$. Then $t_1=0$ as $t_1 \ge -4$ by (\ref{bounds}). Since $a=0$, we have $\sum_{i=2}^{n-1}t_i^2 = -(t_n^2-2t_n) \ge 0$, which implies $t_n = 0, 1, 2$ as $t_n \ge 0$ by (\ref{bounds}). Solving the Diophantine equations (\ref{lem1E01}) and (\ref{a}) for each $t_n$, and using (\ref{bounds}), we obtain all possibilities for $(t_2, \ldots, t_{n-1})$ and hence $\deg(H)$ as in Table \ref{tab1}. (In Tables \ref{tab1}--\ref{tab4} an exponent under the column $(t_2, \ldots, t_{n-1})$ indicates the number of times the corresponding value appears in this sequence. For example, $-1^2$ means that $-1$ appears twice.)

\medskip
\noindent\emph{Case 2.} $a=5$. Then $t_1=-1$ as $t_1 \ge -4$ by (\ref{bounds}). Since $a=5$, we have $\sum_{i=2}^{n-1}t_i^2 = 5-(t_n^2-2t_n) \ge 0$, which implies $t_n = 0, 1, 2, 3$ as $t_n \ge 0$ by (\ref{bounds}). Again, by using (\ref{bounds}), (\ref{lem1E01}) and (\ref{a}), we obtain all possibilities for $(t_2, \ldots, t_{n-1})$ and $\deg(H)$ as shown in Table \ref{tab2}.

\medskip
\noindent\emph{Case 3.} $a=8$. Then $t_1=-2$ or $t_1=-4$, and so (\ref{lem1E01}) gives $\sum_{i=2}^{n}t_i=2$ or $4$, respectively. Since $\sum_{i=2}^{n-1}t_i^2 = 8 - (t_n^2-2t_n) \ge 0$ and $t_n \ge 0$, in each case we have $t_n = 0, 1, 2, 3, 4$. So we have ten combinations in total. Using (\ref{bounds}), (\ref{lem1E01}) and (\ref{a}), we obtain all possibilities for $(t_2, \ldots, t_{n-1})$ and $\deg(H)$ as shown in Table \ref{tab3}.

\medskip
\noindent\emph{Case 4.} $a=9$. Then $t_1=-3$ and so $\sum_{i=2}^{n}t_i=3$. Since $\sum_{i=2}^{n-1}t_i^2 = 9 - (t_n^2-2t_n) \ge 0$ and $t_n \ge 0$, we have $t_n = 0, 1, 2, 3, 4$. Again, by using (\ref{bounds}), (\ref{lem1E01}) and (\ref{a}), we obtain all possibilities for $(t_2, \ldots, t_{n-1})$ and $\deg(H)$ as shown in Table \ref{tab4}. \qed\end{proof}

\begin{table}
\begin{center}
  \begin{tabular}{c|c|l|l}
\hline
$t_1$  & $t_n$  & $(t_2, \ldots, t_{n-1})$ & $\deg(H)$ \\  \hline
$0$     &  $0$    &  $(0^{n-2})$             & $(5,2^{n-2},1)$ \\ \hline
$0$     &  $1$    &  $(-1^1, 0^{n-3})$  & $(5,2^{n-2},1)$ \\ \hline
$0$     &  $2$    &  Infeasible  &  \\ \hline
  \end{tabular}
  \caption{$a=0$}
\label{tab1}
  \end{center}
\end{table}

\begin{table}
\begin{center}
  \begin{tabular}{c|c|l|l}
\hline
$t_1$  & $t_n$  & $(t_2, \ldots, t_{n-1})$ & $\deg(H)$ \\  \hline
$-1$    &  $0$    &  $(2^1, -1^1, 0^{n-4})$, $(1^3, -1^2, 0^{n-7})$  & $(4^2,2^{n-4},1^2)$, $(4,3^3,2^{n-7},1^3)$ \\ \hline
$-1$    &  $1$    &  $(1^3, -1^3, 0^{n-8})$, $(2^1, -1^2, 0^{n-5})$  & $(4,3^3,2^{n-7},1^3)$, $(4^2,2^{n-4},1^2)$ \\ \hline
$-1$    &  $2$    & $(1^2, -1^3, 0^{n-7})$   & $(4,3^3,2^{n-7},1^3)$ \\ \hline
$-1$    &  $3$    & $(-1^2, 0^{n-4})$          & $(4^2,2^{n-4},1^2)$ \\ \hline
  \end{tabular}
  \caption{$a=5$}
\label{tab2}
  \end{center}
\end{table}

\begin{table}
\begin{center}
  \begin{tabular}{c|c|l|l}
\hline
$t_1$  & $t_n$  & $(t_2, \ldots, t_{n-1})$ & $\deg(H)$ \\  \hline
$-2$    &  $0$    &   $(2^1, 1^2, -1^2, 0^{n-7})$, $(1^5, -1^3, 0^{n-10})$ & $(4,3^3,2^{n-7},1^3)$, $(3^6,2^{n-10},1^4)$ \\ \hline
$-2$    &  $1$    & $(2^1, 1^2, -1^3, 0^{n-8})$, $(1^5, -1^4, 0^{n-11})$ & $(4,3^3,2^{n-7},1^3)$, $(3^6,2^{n-10},1^4)$ \\ \hline
$-2$    &  $2$    & $(2^1, 1^1, -1^3, 0^{n-7})$, $(1^4, -1^4, 0^{n-10})$ & $(4,3^3,2^{n-7},1^3)$, $(3^6,2^{n-10},1^4)$ \\ \hline
$-2$    &  $3$    & $(1^2, -1^3, 0^{n-7})$ & $(4,3^3,2^{n-7},1^3)$ \\ \hline
$-2$    &  $4$    & Infeasible &  \\ \hline
$-4$    &  $0$    & $(2^2, 0^{n-4})$, $(2^1, 1^3, -1^1, 0^{n-7})$ & $(4^2,2^{n-4},1^2)$, $(4,3^3,2^{n-7},1^3)$ \\
          &            & $(1^6, -1^2, 0^{n-10})$ & $(3^6,2^{n-10},1^4)$ \\ \hline
$-4$    &  $1$    & $(3^1, 0^{n-4})$, $(2^2, -1^1, 0^{n-5})$ & $(5,2^{n-2},1)$, $(4^2,2^{n-4},1^2)$ \\
         &            & $(2^1, 1^3, -1^2, 0^{n-8})$, $(1^6, -1^3, 0^{n-11})$ & $(4,3^3,2^{n-7},1^3)$, $(3^6,2^{n-10},1^4)$ \\ \hline
$-4$    &  $2$    & $(2^1,1^2,-1^{2},0^{n-7})$, $(1^5,-1^3,0^{n-10})$ & $(4,3^3,2^{n-7},1^3)$, $(3^6,2^{n-10},1^4)$ \\ \hline
$-4$    &  $3$    & $(2^1,-1^{1},0^{n-4})$, $(1^3,-1^{2},0^{n-7})$ & $(4^2,2^{n-4},1^2)$, $(4,3^3,2^{n-7},1^3)$ \\ \hline
$-4$    &  $4$    & $(0^{n-2})$ & $(5,2^{n-2},1)$ \\ \hline
  \end{tabular}
  \caption{$a=8$}
\label{tab3}
  \end{center}
\end{table}

\begin{table}
\begin{center}
  \begin{tabular}{c|c|l|l}
\hline
$t_1$  & $t_n$  & $(t_2, \ldots, t_{n-1})$ & $\deg(H)$ \\  \hline
$-3$    &  $0$    & $(3^1, 0^{n-3})$, $(2^2, -1^1, 0^{n-5})$ & $(5,2^{n-2},1)$, $(4^2,2^{n-4},1^2)$ \\
          &            & $(2^1, 1^3, -1^2, 0^{n-8})$, $(1^6, -1^3, 0^{n-11})$ & $(4,3^3,2^{n-7},1^3)$, $(3^6,2^{n-10},1^4)$ \\ \hline
$-3$    &  $1$    & $(3^1, -1^1, 0^{n-4})$, $(2^2, -1^2, 0^{n-6})$ & $(5,2^{n-2},1)$, $(4^2,2^{n-4},1^2)$ \\
          &            & $(2^1, 1^3, -1^3, 0^{n-9})$, $(1^6, -1^4, 0^{n-12})$ & $(4,3^3,2^{n-7},1^3)$, $(3^6,2^{n-10},1^4)$ \\ \hline
$-3$    &  $2$    & $(2^1, 1^2, -1^3, 0^{n-8})$, $(1^5, -1^4, 0^{n-11})$ & $(4,3^3,2^{n-7},1^3)$, $(3^6,2^{n-10},1^4)$ \\ \hline
$-3$    &  $3$    & $(2^1, -1^2, 0^{n-5})$, $(1^3, -1^3, 0^{n-8})$ & $(4^2,2^{n-4},1^2)$, $(4,3^3,2^{n-7},1^3)$ \\ \hline
$-3$    &  $4$    & $(-1^1, 0^{n-3})$ & $(5,2^{n-2},1)$ \\ \hline
  \end{tabular}
  \caption{$a=9$}
\label{tab4}
  \end{center}
\end{table}

\begin{lem}\label{lem2}
Suppose the propeller graph $G$ has at most one triangle. If a graph $H$ is $L$-cospectral with $G$, then $\mathrm{deg}(H)=(5,2^{n-2},1)$.
\end{lem}

\begin{proof}
Since $H$ is $L$-cospectral with $G$, by Lemma \ref{lem1},
\[\mathrm{deg}(H) = (5,2^{n-2},1),(4^2,2^{n-4},1^2),(4,3^3,2^{n-7},1^3),\;\, \mbox{or}\;\, (3^6,2^{n-10},1^4).\]
In view of the formula for $l_3$ in Lemma \ref{Lcoefficients}, we obtain
\begin{eqnarray}
6n_3(G)-\sum_{v \in V(G)}d_G(v)^3=6n_3(H)-\sum_{v \in V(H)}d_H(v)^3.\label{lem2E01}
\end{eqnarray}
Note that $n_3(G) = 1$ or $0$ since $G$ contains at most one triangle by our assumption.

\medskip
\noindent\emph{Case 1.} $\mathrm{deg}(H)=(4^2,2^{n-4},1^2)$. In this case by (\ref{lem2E01}) we have
\begin{eqnarray}
6n_3(G)-(8n+110)=6n_3(H)-(8n+98), \label{lem2E02}
\end{eqnarray}
and so $n_3(H)=-1$ or $-2$, depending on whether $n_3(G) = 1$ or $0$. This is a contradiction because $n_3(H) \ge 0$ by its definition.

\medskip
\noindent\emph{Case 2.} $\mathrm{deg}(H)=(4,3^3,2^{n-7},1^3)$. By (\ref{lem2E01}), we have
\begin{eqnarray}
6n_3(G)-(8n+110)=6n_3(H)-(8n+92),\label{lem2E03}
\end{eqnarray}
which leads to $n_3(H)=-2$ or $-3$, again a contradiction.

\medskip
\noindent\emph{Case 3.} $\mathrm{deg}(H)=(3^6,2^{n-10},1^4)$. Then (\ref{lem2E01}) implies
\begin{eqnarray}
6n_3(G)-(8n+110)=6n_3(H)-(8n+86).\label{lem2E04}
\end{eqnarray}
This leads to $n_3(H)=-3$ or $-4$, which is a contradiction.

Therefore, the only possibility is $\mathrm{deg}(H)=(5,2^{n-2},1)$.
\qed\end{proof}

\begin{lem}\label{lem3}
Suppose the propeller graph $G$ has two triangles. If a graph $H$ is $L$-cospectral to $G$, then $\mathrm{deg}(H)=(5,2^{n-2},1)$ or $(4^2,2^{n-4},1^2)$, and the latter occurs only when $H$ is triangle-free.
\end{lem}

\begin{proof}
The proof is straightforward by using (\ref{lem2E02}), (\ref{lem2E03}) and (\ref{lem2E04}).
\qed\end{proof}

\begin{figure}[here]
\centering
\vspace{-0.3cm}
\includegraphics*[height=6cm]{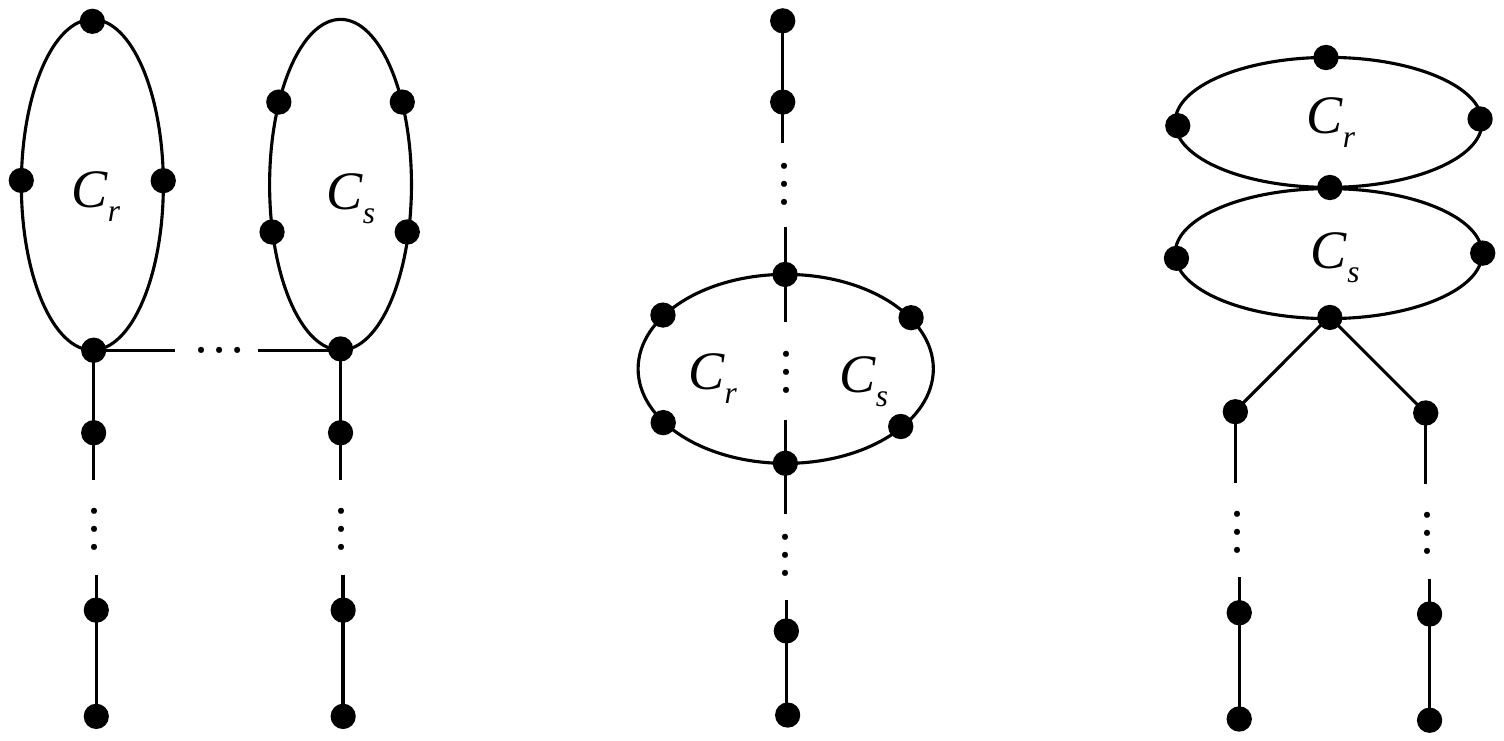}
\vspace{-1.8cm}
\caption{\small Proof of Theorem \ref{thm1}: possible cases for $H$.}
\label{f2}
\end{figure}

Now we are ready to prove Theorem \ref{thm1}.

\bigskip

\begin{Tproof} \textbf{of Theorem \ref{thm1}}.~~Let $G$ be a propeller graph with at most one triangle. Suppose $H$ is $L$-cospectral with $G$. By Lemma \ref{lem2}, $\mathrm{deg}(H)=(5,2^{n-2},1)$. Since $H$ is connected by (c) of Lemma \ref{spectrum}, it follows that $H$ must be a propeller graph. By Lemma \ref{lem4}, we conclude that $H$ and $G$ are isomorphic.

Let $G$ be a propeller graph with two triangles; that is, $p=q=3$. Suppose $H$ is $L$-cospectral with $G$. By Lemma \ref{lem3}, $\mathrm{deg}(H)=(5,2^{n-2},1)$ or $(4^2,2^{n-4},1^2)$, and in the latter case $H$ is triangle-free. In the case when $\mathrm{deg}(H)=(5,2^{n-2},1)$, similar to the argument in the first paragraph, it is straightforward to show that $H$ and $G$ are isomorphic.

Consider the case $\mathrm{deg}(H)=(4^2,2^{n-4},1^2)$, where $H$ is triangle-free. Since $H$ is connected by (c) of Lemma \ref{spectrum}, there are three possibilities for $H$ as shown in Fig. \ref{f2}. However, since $H$ is triangle-free (that is, $r, s \ge 4$), in each case $H$ has more than 9 spanning trees, whilst $G$ has exactly $pq=9$ spanning trees. This contradicts (d) of Lemma \ref{spectrum}.

Therefore, $H$ is isomorphic to $G$ and the proof is complete.
\qed \end{Tproof}

\section{Proof of Theorem \ref{thm2}}
\label{sec:Q}

Throughout this section $G$ is a propeller graph with $n=p+q+k-1$ vertices as shown in Fig. \ref{f1}. Applying Lemma \ref{cycle} to $G$, with $u$ the vertex of degree 5 in $G$, we obtain
\begin{eqnarray}\label{relation}
 \phi(A(G);x)&=&x\phi(A(P_{p-1}))\phi(A(P_{q-1}))\phi(A(P_{k}))-2\phi(A(P_{p-2}))\phi(A(P_{q-1}))\phi(A(P_{k})) \nonumber\\
           & &-2\phi(A(P_{p-1}))\phi(A(P_{q-2}))\phi(A(P_{k}))-\phi(A(P_{p-1}))\phi(A(P_{q-1}))\phi(A(P_{k-1}))\nonumber\\
           & &-2\phi(A(P_{p-1}))\phi(A(P_{k}))-2\phi(A(P_{q-1}))\phi(A(P_{k})).
\end{eqnarray}

The next lemma follows from (\ref{relation}) and $\phi(A(P_n),2)=n+1$ \cite{kn:Ramezani09}.
\begin{lem}\label{A2equal}
 $\phi(A(G);2)=-(3k+2)pq.$
\end{lem}

In \cite{kn:Ramezani09}, the adjacency characteristic polynomial of $P_n$ with $n\geq1$ is given as follows:
\begin{eqnarray}\label{rrecurrence}
             \phi(A(P_n);x)=\frac{y^{2n+2}-1}{y^{n+2}-y^n},
           \end{eqnarray}
where $y$ satisfies $y^2-x y+1=0$ with $x\neq2$. Substituting (\ref{rrecurrence}) into (\ref{relation}), by using Maple, we obtain
\begin{eqnarray}\label{AUn1}
y^n(y^2-1)^3\phi(A(G))+1-4y^2-y^{2n+6}+4y^{2n+4}=f_A(p,q,k;y),
\end{eqnarray}
where $n=p+q+k-1$ and
\begin{eqnarray*}
\begin{array}{lllll}
   f_A(p,q,k;y)=&-2y^{4+2k+p+2q}    &-2y^{4+2k+q+2p}   &+2y^{2k+2+p+2q}    &+2y^{2k+2+q+2p}\\
                &+3y^{2p+2q}        &+2y^{2+p+2q}      &+2y^{2+q+2p}       &-2y^{p+2q}\\
                &-2y^{q+2p}         &-2y^{2+2p}        &-2y^{2+2q}         &-y^{2p}-y^{2q}\\
                &+y^{2k+4+2p}       &+y^{2k+4+2q}      &+2y^{2k+2+2p}      &+2y^{2k+2+2q}    \\
                &+2y^{2k+4+p}       &+2y^{2k+4+q}      &-2y^{2k+2+p}       &-2y^{2k+2+q} \\
                &-2y^{2+p}          &-2y^{2+q}         &+2y^p+2y^q         &-3y^{2k+4}.
\end{array}
\end{eqnarray*}
\begin{lem}\label{Anonpropepeller}
No two non-isomorphic propeller graphs are $A$-cospectral.
\end{lem}
\begin{proof}
Let $G'$ be a propeller graph with order $n'=p'+q'+k'-1$. Suppose that $G'$ and
$G$ are $A$-cospectral. Without loss of generality, we may assume $p\geq q$ and $p'\geq q'$. Since cospectral graphs have the same order, we have
\begin{eqnarray}\label{Asamenumber}
 p+q+k=p'+q'+k'.
\end{eqnarray}
Lemma \ref{A2equal} implies
\begin{eqnarray}\label{A2numberequal}
 (3k+2)pq=(3k'+2)p'q'.
\end{eqnarray}
By (\ref{AUn1}), we have
\begin{eqnarray}\label{Afequal}
 f_A(p,q,k;y)=f_A(p',q',k';y).
\end{eqnarray}
The term in $f_A(p,q,k;y)$ with the smallest exponent is $-3y^{2k+4}$ or $2y^{q}$, and similarly for $f_A(p',q',k';y)$. From (\ref{Afequal}) we have either $-3y^{2k+4}=-3y^{2k'+4}$ or $2y^{q}=2y^{q'}$. In the former case, we have $k=k'$, and so $(p,q)=(p',q')$ by (\ref{Asamenumber}) and (\ref{A2numberequal}). In the latter case, we have $q=q'$. Suppose $k\neq k'$. Without loss of generality, let $k'= k+i$ where $i \ge 1$. Substituting back into (\ref{Asamenumber}), we get $p'=p-i$, and then
$(3i+3k+2-3p)i=0$, via expressing $p',q',k'$ by $p,q,k$ and $i$ in (\ref{A2numberequal}). Clearly, $3i+3k+2-3p\neq0$, a contradiction. So, $k=k'$, and then $p=p'$.  Therefore, $G$ and $G'$ are isomorphic in each case.
\qed\end{proof}

Since the subdivision graph of a propeller graph $G$ is also a propeller graph, Lemmas \ref{Anonpropepeller} and \ref{QA} immediately imply the following result.

\begin{lem}\label{Anonprop}
No two non-isomorphic propeller graphs are $Q$-cospectral.
\end{lem}

\begin{lem}\label{interlacingthmAA}
Let $G$ be a propeller graph. Then $\lambda_2(G)<2$.
\end{lem}

\begin{proof} Let $u$ be the vertex of degree 4 in $G$.
By the Interlacing Theorem \cite{kn:Godsil01} for the $A$-spectrum, we obtain
\begin{eqnarray*}
\lambda_2(G)\leq\lambda_1(G-u)=
\lambda_1\left(P_{q-1}\cup P_{p-1}\cup P_{k}\right)<2,
\end{eqnarray*}
where the last inequality holds because the largest eigenvalue for the $A$-spectrum of a path is less than 2. \qed\end{proof}

\begin{cor}\label{interlacingthm}
Let $G$ be a propeller graph. Then $\lambda_2\left(\mathcal
{S}(G)\right)<2$.
\end{cor}

\begin{figure}[here]
\centering
\vspace{-0.7cm}
\includegraphics*[height=6cm]{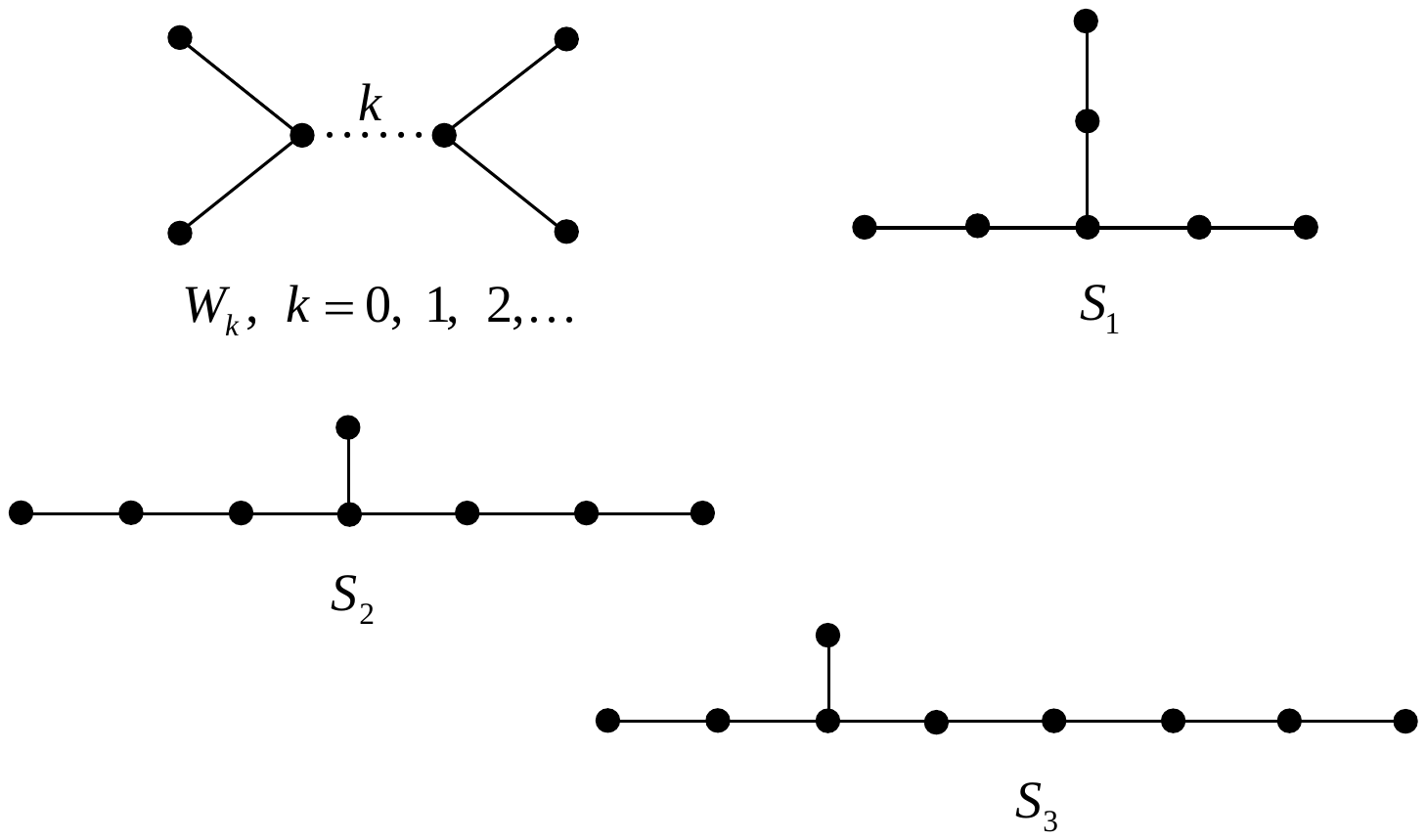}
\vspace{-1.4cm}
\caption{Smith graphs $W_k$, $S_1$, $S_2$ and $S_3$.}\label{smith}
\end{figure}

A connected graph which satisfies $\lambda_1=2$ is called a \emph{Smith graph} \cite{kn:Smith70}. All Smith graphs are known in \cite{kn:Smith70}. They are cycles $C_n$ ($n \ge 3$) and the graphs depicted in Fig. \ref{smith}, where in $W_k$, $k$ is the length of the path joining the middle vertices of the two copies of $P_3$. (Note that $W_0=K_{1,4}$.)

\begin{lem}\label{no2cycles}
Let $H$ be a graph that is $Q$-cospectral with the propeller graph $G$. Then $H$
does not contain two vertex-disjoint cycles as its subgraph.
\end{lem}

\begin{proof}
Since $H$ is $Q$-cospectral with $G$, by Lemma \ref{QA}, $\mathcal {S}(H)$ is $A$-cospectral to $\mathcal {S}(G)$. This together with Corollary \ref{interlacingthm} implies $\lambda_2(\mathcal {S}(H))=\lambda_2(\mathcal {S}(G))<2$. Since the largest eigenvalue for the $A$-spectrum of a cycle is 2, it follows that $\mathcal {S}(H)$ does not contain two vertex-disjoint cycles. Since $\mathcal {S}(H)$ is the subdivision graph of $H$, the same result holds for $H$. \qed\end{proof}

\begin{lem}
\label{Qlem1}
Let $H$ be a graph that is $Q$-cospectral with the propeller graph $G$. Then
\begin{eqnarray}
  \mathrm{deg}(H)&=&(5,2^{n-2},1),(4^2,2^{n-4},1^2),(4,3^3,2^{n-7},1^3), (3^6,2^{n-10},1^4), \nonumber \\
   & & (4,3^2,2^{n-4},0),\;\mbox{or}\,\; (3^5,2^{n-7},1,0). \label{eq:Qlem1}
\end{eqnarray}
\end{lem}
\begin{proof}
Suppose $\mathrm{deg}(H)=(5+t_1,2+t_2,2+t_3,\dots,2+t_{n-1},1+t_n)$. Since the connectivity of $H$ can not be determined by its $Q$-spectrum, $H$ may contain just isolated vertices as its components. Thus
\begin{equation}\label{newtbound}
t_1 \ge -5,\, t_2 \ge -2,\, \ldots,\, t_{n-1} \ge -2,\, t_n \ge -1.
\end{equation}
The rest of the proof is similar to that of Lemma \ref{lem1} and hence we omit details.
 \qed\end{proof}

\begin{lem}\label{no3triangles}
Let $H$ be a graph that is $Q$-cospectral with the propeller graph $G$. Then $H$
is a propeller graph.
\end{lem}

\begin{proof}
Since $H$ is $Q$-cospectral with $G$, by Lemma \ref{Qtriangles},
\begin{eqnarray}
6n_3(G)+\sum_{v \in V(G)}d_G(v)^3=6n_3(H)+\sum_{v \in V(G)}d_H(v)^3.\label{Qlem2E01}
\end{eqnarray}
Since $G$ is a propeller graph, by Lemma \ref{Qlem1}, the degree sequence of $H$ is given in (\ref{eq:Qlem1}). We consider the cases for $\mathrm{deg}(H)$ one by one. Note that $n_3(G) =0$, $1$ or $2$.

\medskip
\noindent\emph{Case 1.} $\mathrm{deg}(H) = (5,2^{n-2},1)$. It is straightforward to show that $H$ is a propeller graph.

\medskip
\noindent\emph{Case 2.} $\mathrm{deg}(H)=(4^2,2^{n-4},1^2)$. In this case, by (\ref{Qlem2E01}) we have $6n_3(G)+(8n+110)=6n_3(H)+(8n+98)$. Hence $n_3(H)=2, 3, 4$ depending on whether $n_3(G)=0, 1, 2$ respectively.

By Lemma \ref{no2cycles} and $\mathrm{deg}(H)=(4^2,2^{n-4},1^2)$, there are three possibilities for $H$ as shown in Fig. \ref{Qf3}. Note that for the $Q$-spectrum the multiplicity of 0 gives the number of bipartite components \cite{kn:Cvetkovic07}. Clearly, for $H_1$, there is an eigenvalue 0 in its $Q$-spectrum, but there is no eigenvalue 0 in the $Q$-spectrum of $G$, since $n_3(G)=1$, that is, $G$ is not bipartite. This is a contradiction, because $G$ and $H$ are not $Q$-cospectral.

If $H$ is isomorphic to $H_2$, then Lemma
\ref{AQAcospectral} implies that the line graphs $\mathcal {L}(G)$ and
$\mathcal {L}(H_2)$ are $A$-cospectral, that is $\sum\limits_i\lambda_i(\mathcal
{L}(G))^4= \sum\limits_i\lambda_i(\mathcal {L}(H_2))^4$. However, by Lemma
\ref{closedfourwalk}, this cannot happen by the following computation:
\begin{eqnarray*}
\sum\limits_i\lambda_i(\mathcal {L}(H_2))^4=\left\{\begin{array}{rll}
   310, & & \text{if $l=1$ and $t=1$};\\
  6n+276,& & \text{if $l\geq2$ and $t=1$};\\
  6n+276,& & \text{if $l=1$ and $t\geq2$};\\
  6n+284, & & \text{if $l\geq2$ and $t\geq2$};
  \end{array}\right.
\end{eqnarray*}
\begin{eqnarray}\label{4square}
\sum\limits_i\lambda_i(\mathcal {L}(G))^4=\left\{\begin{array}{rll}
  368, & & \text{if $p=q=4$ and $k=1$};\\
  6n+332,& & \text{if $p=q=4$ and $k\geq2$};\\
  6n+312, & & \text{if $p>q=4$ and $k=1$};\\
  6n+324, & & \text{if $p>q=4$ and $k\geq2$};\\
  6n+304, & & \text{if $p\geq q>4$ and $k=1$};\\
  6n+316, & & \text{if $p\geq q>4$ and $k\geq2$}.\\
  \end{array}\right.
\end{eqnarray}

If $H$ is isomorphic to $H_3$, similarly to the above case, $\sum\limits_i\lambda_i(\mathcal {L}(H_3))^4$ is computed as follows:
\begin{eqnarray*}
\sum\limits_i\lambda_i(\mathcal {L}(H_3))^4=\left\{\begin{array}{rll}
   328, & & \text{if $l=1$ and $t=1$};\\
  6n+300,& & \text{if $l\geq2$ and $t=1$};\\
  6n+300,& & \text{if $l=1$ and $t\geq2$};\\
  6n+308, & & \text{if $l\geq2$ and $t\geq2$}.
  \end{array}\right.
\end{eqnarray*}
Again, $\sum\limits_i\lambda_i(\mathcal
{L}(G))^4\neq\sum\limits_i\lambda_i(\mathcal {L}(H_3))^4$, a contradiction.

\begin{figure}[here]
\centering
\vspace{-0.3cm}
\includegraphics*[height=6cm]{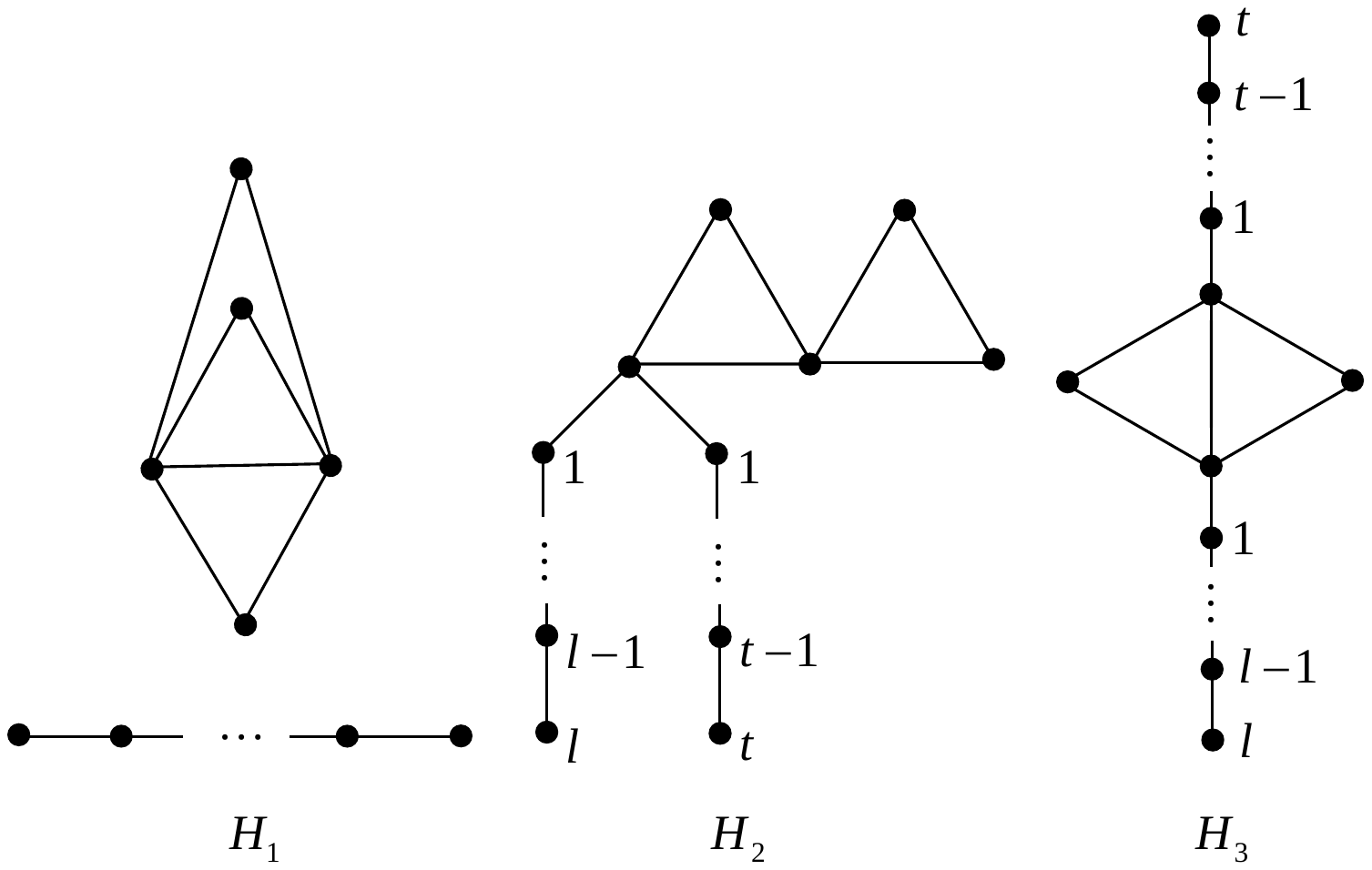}
\vspace{-1.2cm}
\caption{\small Proof of Lemma \ref{no3triangles}: Case 2.}
\label{Qf3}
\end{figure}

\begin{figure}[here]
\centering
\vspace{-0.3cm}
\includegraphics*[height=6cm]{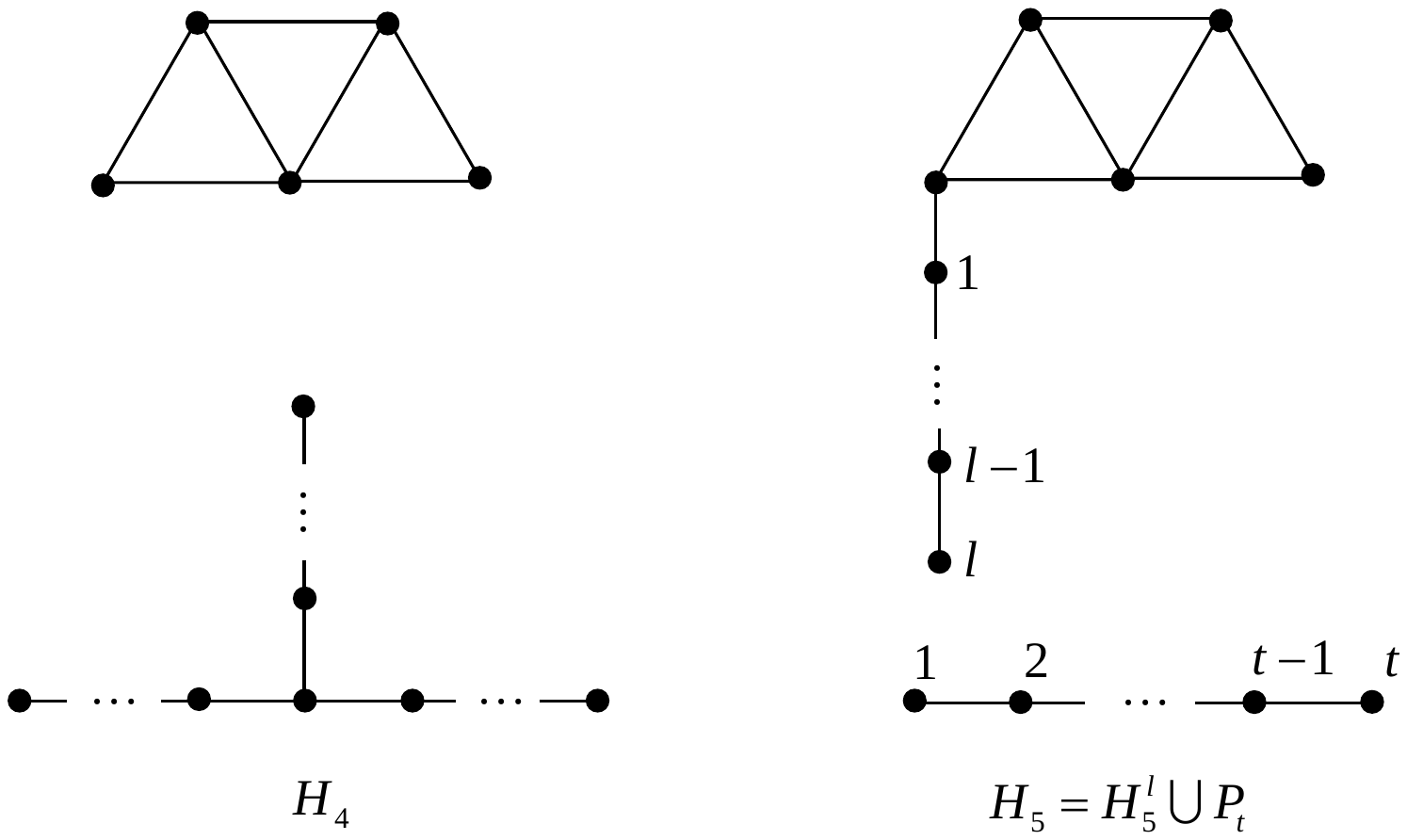}
\vspace{-1.4cm}
\caption{\small Proof of Lemma \ref{no3triangles}: Case 3.}
\label{Qf4}
\end{figure}
\medskip
\noindent\emph{Case 3.} $\mathrm{deg}(H)=(4,3^3,2^{n-7},1^3)$. In this case, by (\ref{Qlem2E01}), we have $6n_3(G)+(8n+110)=6n_3(H)+(8n+92)$. Hence $n_3(H)=3, 4, 5$ depending on whether $n_3(G)=0, 1, 2$ respectively. Again, by Lemma \ref{no2cycles} and $\mathrm{deg}(H)=(4,3^3,2^{n-7},1^3)$, there are two possibilities for $H$ as shown in Fig. \ref{Qf4}. If $H$ is isomorphic to $H_4$, then $\mathcal {S}(H)$ contains a subgraph isomorphic to a disjoint union of a cycle and the Smith graph $S_1$. This contradicts the fact $\lambda_2(\mathcal {S}(H))=\lambda_2(\mathcal {S}(G))<2$.

If $H$ is isomorphic to $H_5$, then Lemma
\ref{AQAcospectral} implies that the line graphs $\mathcal {L}(G)$ and
$\mathcal {L}(H_5)$ are $A$-cospectral, that is $\sum\limits_i\lambda_i(\mathcal
{L}(G))^4= \sum\limits_i\lambda_i(\mathcal {L}(H_5))^4$. By Lemma
\ref{closedfourwalk}, we have
\begin{eqnarray*}
\sum\limits_i\lambda_i(\mathcal {L}(H_5))^4=\left\{\begin{array}{rll}
   368, & & \text{if $l=1$ and $t=2$};\\
  6n+316,& & \text{if $l=1$ and $t\geq3$};\\
  6n+324, & & \text{if $l\geq2$ and $t=2$};\\
  6n+320, & & \text{if $l\geq2$ and $t\geq3$};
  \end{array}\right.
\end{eqnarray*}
By the above computation and (\ref{4square}), there exist three equal cases:

\medskip
\noindent\emph{Case 3.1.} 368: $H_5$ with $l=1$, $t=2$ and $G$ with $p=q=4$, $k=1$. With the help of Maple, we have
\begin{eqnarray*}
  \phi(Q(H_5);x) &=& x^8-18x^7+128x^6-468x^5+948x^4-1054x^3+584x^2-120x  ;\\
  \phi(Q(G);x) &=&  x^8-18x^7+128x^6-468x^5+948x^4-1056x^3+592x^2-128x.
\end{eqnarray*}
Clearly, $\phi(Q(H_5))\neq\phi(Q(G))$, a contradiction.

\medskip
\noindent\emph{Case 3.2.} $6n+316$: $H_5$ with $l=1$, $t\geq3$ and $G$ with $p\geq q>4$, $k\geq2$. Note that $H_5$ contains an eigenvalue 0 in its $Q$-spectrum. Then $p$ and $q$ must be even numbers no less than 6. By Lemma \ref{QTUgraph}, we have $q_{n-1}(G)=(-1)^{n-1}pqn$, and $q_{n-1}(H_5)=(-1)^{n-1}(60n-360)$. Then $q_{n-1}(G)=q_{n-1}(H_5)$ implies $36n=60n-360$ or $48n=60n-360$, since $q_{n-1}(G)>q_{n-1}(H_5)$ with $p\geq q\geq8$. In the former case, we have $n=15$. That is,  $H_5$ has 15 vertices with $l=1$, $t=9$, and $G$ has 15 vertices with $p=q=6$, $k=4$.  Note that for a bipartite graph $G'$, $\phi(Q(G'))=\phi(L(G'))$ \cite{kn:Cvetkovic07}. Thus, $\phi(Q(H_5))=\phi(Q(H_5^1))\phi(L(P_{9}))$, and $\phi(Q(G))=\phi(L(G))$. By Maple,  we obtain
\begin{eqnarray}\label{QH5pol1}
  \phi(Q(H_5^1);x)=x^6-16x^5+96x^4-276x^3+396x^2-262x+60.
\end{eqnarray}
Substituting $x=(y+1)^2/y$ into (\ref{QH5pol1}), then plugging (\ref{LaplacianPolwhole}) and (\ref{QH5pol1}) into the expression of $\phi(Q(H_5))$, and with the help of Maple, we obtain
\[y^{15}(y-1)^3(y+1)^2\phi(Q(H_5))+1-3y-4y^2+4y^{33}+3y^{34}-y^{35}=f_Q(H_5;y),\]
where
\begin{eqnarray*}
  f_Q(H_5;y) &=& 2y^{30}+2y^{29}+2y^{28}+2y^{25}+2y^{24}+2y^{23}-4y^{20}-3y^{19}+y^{18} \\
           & &-y^{17}+3y^{16}+4y^{15}-2y^{12}-2y^{11}-2y^{10}-2y^7-2y^6-2y^5.
\end{eqnarray*}
Substituting $p=q=6$ and $k=4$ back into (\ref{Ppolywhole}),  we have
\begin{eqnarray*}
  f_L(6,6,4;y) &=& -4y^{29}+4y^{27}+3y^{25}+3y^{24}+2y^{23}+6y^{22}+4y^{21}+4y^{20}-4y^{18} \\
           & &+4y^{17}-4y^{15}-4y^{14}-6y^{13}-2y^{12}-3y^{11}-3y^{10}-4y^8+4y^6.
\end{eqnarray*}
Thus, $f_Q(H_5;y)\neq f_L(6,6,4;y)$. This contradicts $\phi(Q(H_5))=\phi(Q(G))$.

In the latter case, we have $n=30$.  That is,  $H_5$ has 30 vertices with $l=1$, $t=24$, and $G$ has 30 vertices with $p=8$, $q=6$, $k=17$. Using the similar method to the former case, we have $f_Q(H_5;y)\neq f_L(8,6,17;y)$, which also contradicts $\phi(Q(H_5))=\phi(Q(G))$.

\medskip
\noindent\emph{Case 3.3.} $6n+324$: $H_5$ with $l\geq2$ and $t=2$ and $G$ with $p>q=4$ and $k\geq2$. Similarly to Case 3.2, $p$ must be even numbers no less than 6, and Lemma \ref{QTUgraph} implies that  $q_{n-1}(G)=(-1)^{n-1}4pn$, and $q_{n-1}(H_5)=(-1)^{n-1}120$. Clearly, $q_{n-1}(G)\neq q_{n-1}(H_5)$, a contradiction.

\medskip
\noindent\emph{Case 4.} $\mathrm{deg}(H)=(3^6,2^{n-10},1^4)$. In this case, (\ref{Qlem2E01}) yields $6n_3(G)+(8n+110)=6n_3(H)+(8n+86)$. Hence $n_3(H)=4, 5, 6$ depending on whether $n_3(G)=0, 1, 2$ respectively. By Lemma \ref{no2cycles}, there is no feasible $H$ satisfying $\mathrm{deg}(H)=(3^6,2^{n-10},1^4)$.

\medskip
\noindent\emph{Case 5.} $\mathrm{deg}(H)=(4,3^2,2^{n-4},0)$. In this case there is an eigenvalue 0 in the $Q$-spectrum of $H$. This implies that $G$ must be bipartite and so $n_3(G)=0$. By (\ref{Qlem2E01}), we have
$6n_3(G)+(8n+110)=6n_3(H)+(8n+86)$, which gives $n_3(H)=4$.
Clearly, by Lemma \ref{no2cycles}, there are no feasible $H$ satisfying $\mathrm{deg}(H)=(4,3^2,2^{n-4},0)$.

\medskip
\noindent\emph{Case 6.} $\mathrm{deg}(H)=(3^5,2^{n-7},1,0)$. Similar to Case 5, we have $n_3(G)=0$. Again, by (\ref{Qlem2E01}), we have $n_3(H)=4$. Lemma \ref{no2cycles} implies that there is no feasible $H$ satisfying $\mathrm{deg}(H)=(3^5,2^{n-7},1,0)$.

The proof is complete.
\qed\end{proof}

\begin{Tproof} \textbf{of Theorem \ref{thm2}.}~~
The result follows from Lemmas \ref{Anonprop} and \ref{no3triangles} immediately.
\qed \end{Tproof}

\section{Conclusion}

In this paper, we proved that any propeller graph is determined by its $L$-spectrum as well as its $Q$-spectrum. Along the way we showed that no two non-isomorphic propeller graphs are $A$-cospectral (Lemma \ref{Anonpropepeller}). We expect that this result could be used to prove some propeller graphs are $A$-DS. On the other hand, not every propeller graph is determined by its $A$-spectrum. For example, in \cite[pp.12]{kn:Biggs93} and \cite[pp.1226]{kn:Lu09}, two $A$-cospectral mates are given. And we expect that there are more graphs that are $A$-cospectral with propeller graphs. It would be an interesting question to characterize which graphs are $A$-cospectral with propeller graphs.

\bigskip

\noindent\textbf{Acknowledgements}\vspace{0.3cm}

We appreciate the anonymous referees for their comments and suggestions. X. Liu is supported by MIFRS and MIRS of the University of Melbourne and the Natural Science Foundation of China (No.11361033). S. Zhou is supported by a Future Fellowship (FT110100629) of the Australian Research Council.


\begin{thebibliography}{99}
\setlength{\parskip}{0pt} \addtolength{\itemsep}{-4pt}
\footnotesize{


\bibitem{kn:Biggs93} N.L. Biggs, Algebraic Graph Theory (Second Edition), Cambridge
University Press, 1993.

\bibitem{kn:Boulet08} R. Boulet, B. Jouve, The lollipop graph is determined by its spectrum, The Electronic Journal of Combinatorics 15 (2008), R74

\bibitem{kn:Boulet09} R. Boulet, Spectral characterizations of sun graphs and broken sun graphs, Discrete Math. Theor. Comput. Sci. 11 (2009) 149--160.

\bibitem{kn:Cvetkovic87}  D.M. Cvetkovi\'{c}, P. Rowlinson, Spectra
of unicyclic graph, Graphs and Combinatorics 3 (1987) 7--23.

\bibitem{kn:Cvetkovic95} D.M. Cvetkovi\'{c}, M. Doob, H. Sachs, Spectra of Graphs - Theory and
Applications, Third edition, Johann Ambrosius Barth. Heidelberg. 1995.


\bibitem{kn:Cvetkovic07} D.M. Cvetkovi\'{c}, P. Rowlinson, S.K.
Simi\'{c}, Signless Laplacians of finite graphs, Linear Algebra
Appl. 423 (2007) 155--171.

\bibitem{kn:Cvetkovic08} D.M. Cvetkovi\'{c}, New theorems for signless Laplacians eigenvalues, Bull. Acad. Serbe Sci. Arts,
Cl. Sci. Math. Natur., Sci. Math. 137(33) (2008), 131--146.

\bibitem{kn:Cvetkovic09} D.M. Cvetkovi\'{c}, S.K.
Simi\'{c}, Towards a spectral theory of graphs based on signless
Laplacian, I, Publ. Inst. Math. (Beograd) (N.S.) 85 (99) (2009)
1--15.

\bibitem{kn:Cvetkovic10} D.M. Cvetkovi\'{c}, P. Rowlinson, H. Simi\'{c}, An Introduction to the Theory of Graph Spectra, Cambridge University Press, Cambridge, 2010.


\bibitem{kn:vanDam03}
E.R. van Dam, W.H. Haemers, Which graphs are determined by their
spectrum?, Linear Algebra Appl. 373 (2003) 241--272.

\bibitem{kn:vanDam09}
E.R. van Dan, W.H. Haemers, Developments on spectral
characterizations of graphs, Discrete Math. 309 (2009) 576--586.



\bibitem{kn:Gunthard56} Hs.H. G\"{u}nthard, H. Primas, Zusammenhang von Graphtheorie und Mo-Theotie von Molekeln mit Systemen konjugierter Bindungen, Helv. Chim. Acta 39 (1956) 1645--1653.

\bibitem{kn:Garey79} M.R. Garey, D.S. Johnson, Computers and Intractability: A Guide to the Theory of NP-Completeness, W. H. Freeman and Company, 1979.


\bibitem{kn:Godsil01} C. Godsil, G. Royle, Algebraic Graph Theory, Springer-Verlag, Inc., New York,
2001, pp. 193--194.

\bibitem{kn:Guo05}
J.-M. Guo, On the second largest Laplacian eigenvalue of trees,
Linear Algebra Appl. 404 (2005) 251--261.

\bibitem{kn:Guo08}
J.-M. Guo, A conjecture on the algebraic connectivity of connected
graphs with fixed girth, Discrete Math. 308 (2008) 5702--5711.


\bibitem{kn:Haemers08} W.H. Haemers, X.-G. Liu, Y.-P. Zhang,  Spectral
characterizations of lollipop graphs, Linear Algebra Appl. 428
(2008) 2415--2423.

\bibitem{kn:Kelmans65} A.K. Kelmans, The number of trees in a graph I, II. Automation and
Remote Control 26 (1965) 2118-2129  and 27 (1966) 233--241,
translated from Avtomatika i Telemekhanika 26 (1965) 2194--2204  and
27 (1966) 56--65  [in Russian].
 
 
\bibitem{kn:LiuM10} M.-H. Liu, B.-L. Liu, Some results on the Laplacian spectrum, Computers and Mathematics with Applications 59 (2010) 3612--3616.
    

\bibitem{kn:Wangliu10} X.-G. Liu, S.-J. Wang, Laplacian spectral characterization of some graph products, Linear Algebra Appl. 437 (2012) 1749--1759.
    
    
\bibitem{kn:Lu09} P.-L. Lu, X.-G. Liu, Z.-T. Yuan, X.-Y, Yong, Spectral characterizations of sandglass graphs, Applied Mathematics Letters 22 (2009) 1225--1230.

\bibitem{kn:Mirzakhah10} M. Mirzakhah, D. Kiani, The sun graph is determined by its signless Laplacian spectrum, Electronic Journal of Linear Algebra 20 (2010) 610--620



\bibitem{kn:Oliveira02}
C.S. Oliveira, N.M.M. de Abreu, S. Jurkiewilz, The characteristic polynomial of the Laplacian of graphs in $(a, b)$-linear cases, Linear Algebra Appl. 365 (2002) 113--121.

\bibitem{kn:Omidi07} G.R. Omidi, K. Tajbakhsh, Starlike trees are determined by their Laplacian spectrum, Linear Algebra Appl. 422 (2007) 654--658.


\bibitem{kn:Ramezani09}F. Ramezani, N. Broojerdian, B. Tayfeh-Rezaie, A note on the spectral characterization of $\theta$-graphs, Linear Algebra Appl. 431 (2009) 626--632.


\bibitem{kn:Smith70} J.H. Smith, Some properties of spectrum of
graphs, in: R. Guy, H. Hanani, N. Sauer, J. Sch\"{o}nheim (Eds.),
Combinatorial Structures and Their Applications, Gordon and Breach,
Science Publ., Inc., New York, London, Paris, 1970, pp. 403--406.


\bibitem{kn:Simic07}  S.K. Simi\'{c}, Z. Stani\'{c}, Q-integral graphs with edge-degrees
at most five, Discrete Math. 308 (2008) 4625--4634.


\bibitem{kn:Wang10} J.-F. Wang, Q.-X. Huang, F. Belardo, E.M. Li Marzi, On the spectral characterizations of $\infty$-graphs, Discrete Mathematics 310 (2010) 1845--1855.


\bibitem{kn:Zhou12} J. Zhou, C. Bu, Laplacian spectral characterizaiton of some graphs obtained by product operation, Discrete Mathematics 312 (2012) 1591--1595.
}
\end{thebibliography}
\end{document}